\documentclass[11pt,letterpaper]{article}
\usepackage{amsfonts, amsmath, amssymb, amscd, amsthm, graphicx,enumerate}
\usepackage{epsfig, color}
\usepackage{hyperref}

\makeatletter
\def\blfootnote{\xdef\@thefnmark{}\@footnotetext}
\makeatother

\newtheorem{thm}{Theorem}[section]
\newtheorem{cor}[thm]{Corollary}
\newtheorem{lem}[thm]{Lemma}

\newtheorem{prob}[thm]{Problem}

\theoremstyle{definition}

\theoremstyle{remark}

\newfont{\eufm}{eufm10}

\newcommand{\e}{\varepsilon }
\renewcommand{\phi}{\varphi}

\renewcommand{\d }{{\rm d} }

\begin{document}

\title{On acylindrical hyperbolicity of groups with positive first $\ell^2$-Betti number}
\author{D. Osin}

\date{}
\maketitle

\begin{abstract}
We prove that every finitely presented group with positive first $\ell^2$-Betti number that surjects onto $\mathbb Z$ is acylindrically hyperbolic. In particular, this implies acylindrical hyperbolicity of groups of deficiency at least $2$.
\end{abstract}

\section{Introduction}
Recall that a group $G$ is \emph{acylindrically hyperbolic} if it admits a non-elementary acylindrical action on a hyperbolic space; in this context, the action is non-elementary if and only if $G$ is not virtually cyclic and has unbounded orbits \cite{Osi13}. The class of acylindrically hyperbolic groups is very broad and includes non-elementary hyperbolic and relatively hyperbolic groups, all but finitely many mapping class groups of orientable closed surfaced (possibly with punctures), $Out(F_n)$ for $n\ge 2$, most $3$-manifold groups, non-virtually cyclic groups acting properly on proper $CAT(0)$ spaces and containing rank-$1$ elements, and many other examples.

On the other hand, being acylindrically hyperbolic is a rather strong property. In particular, many aspects of the theory of hyperbolic and relatively hyperbolic groups can be generalized in the context of acylindrical hyperbolicity. These include various algebraic, model-theoretic, and analytic properties (see \cite{Osi13} for a more detailed survey), small cancellation theory \cite{Hull}, and group theoretic Dehn surgery \cite{DGO}.

The first $\ell^2$-Betti number, denoted $\beta_1^{(2)}(G)$, is a numerical invariant of $G$ ranging in the interval $[0,\infty ]$.
The study of $\ell^2$-Betti numbers of manifolds was started by Atiyah. The definition was extended to groups in the groundbreaking work of Cheeger and Gromov \cite{CG},  and since then the theory has been further developed by Gaboriau, L\"uck and others (see \cite{Gab,L} and references therein). This work and more recent papers of Peterson-Thom \cite{PT} and Chifan-Sinclair-Udrea \cite{CSU} show that groups with positive first $\ell^2$-Betti numbers share many algebraic and analytic properties with groups defined by various geometric ``negative curvature" type conditions.

For instance, if $G$ is a countable group with $\beta_1^{(2)}(G)>0$, then $G$ has finite amenable radical, is not a direct product of two infinite groups, is not boundedly generated, and is not inner amenable (hence the von Neumann algebra of $G$ does not have property $\Gamma$ of Murray and von Neumann) provided  all non-trivial conjugacy classes of $G$ are infinite; if, in addition, $G$ satisfies a weak version of the Atiyah Conjecture, then it contains non-cyclic free subgroups and its reduced $C^\ast$-algebra is simple and has unique trace \cite{CSU,PT}. All these properties, or their stronger analogues, are also known for acylindrically hyperbolic groups.

At first glance, this might seem to be a plain coincidence. However, it turns out that the classes of acylindrically hyperbolic groups and groups with positive first $\ell^2$-Betti numbers are related more closely than one might think.

\begin{thm}\label{main}
Let $G$ be a finitely presented group. Suppose that $\beta_1^{(2)}(G)>0$ and $G$ maps onto $\mathbb Z$. Then $G$ is acylindrically hyperbolic.
\end{thm}

By the L\"uck Approximation Theorem, every finitely presented residually finite group $G$ with $\beta_1^{(2)}(G)>0$ virtually maps onto $\mathbb Z$. Thus we obtain the following.

\begin{cor}\label{cor1}
Let $G$ be a finitely presented residually finite group with $\beta_1^{(2)}(G)>0$. Then $G$ is virtually acylindrically hyperbolic.
\end{cor}

Note that finite presentability of $G$ is essential here. Indeed there exist finitely generated (but not finitely presented) residually finite torsion groups with positive first $\ell^2$-Betti number \cite{LO}, while acylindrically hyperbolic groups are never torsion \cite{DGO}.

The next application of Theorem \ref{main} is to groups of positive deficiency.  It answers a question asked in \cite[Problem 8.3]{MO}. In the particular case of groups with one relation and at least $3$ generators this result was proved in \cite{MO}. Recall that a group $G$ has deficiency at least $k$ if there exists a finite presentation of $G$ with $k$ more generators that relations.

\begin{cor}\label{cor2}
Every finitely presented group of deficiency at least $2$ is acylindrically hyperbolic.
\end{cor}

Recall that, by the Baumslag-Pride theorem, every finitely presented group of deficiency at least $2$ is large, i.e., has a finite index subgroup that admits an epimorphism onto a non-cyclic free group \cite{BP}. However, this does not guarantee acylindrical hyperbolicity. For example, $F_2\times \mathbb Z$ is large, but not acylindrically hyperbolic.

We list here few immediate consequences of our main theorem. Note that all of them but (a) are new even for groups of deficiency $2$ and, in general, fail for large groups.

\begin{cor}\label{cor3}
Let $G$ be a finitely presented group with $\beta_1^{(2)}(G)>0$ that surjects onto $\mathbb Z$. Then the following hold.
\begin{enumerate}
\item[(a)] $G$ is $SQ$-universal.
\item[(b)] ${\rm dim} \,H_b^{(2)} (G, \ell^p(G))=\infty $ for all $1\le p<\infty$.
\end{enumerate}
If, in addition, $G$ has no non-trivial finite normal subgroups, then:
\begin{enumerate}
\item[(c)] (cf. \cite[Cor. 4.6]{PT}) The reduced $C^\ast$-algebra of $G$ is simple and has unique trace.
\item[(d)] $G$ does not satisfy any non-trivial mixed identity.
\end{enumerate}
\end{cor}

Recall that $G$ is \emph{$SQ$-universal} if  every countable group embeds in a quotient of $G$. In particular, (a) implies that $G$ is ``algebraically large"; for example, it contains non-cyclic free subgroups and has uncountably many normal subgroups. By a \emph{mixed identity} of a group $G$ we mean any element $w\in G\ast \langle x_1, \ldots, x_k\rangle$ such that every homomorphism $G\ast \langle x_1, \ldots, x_k\rangle \to G$ that restricts to the identity map on $G$ takes $w$ to $1$. Not satisfying any non-trivial mixed identity is a rather restrictive conditions. It implies several interesting model-theoretic, geometric, and algebraic properties of $G$. For details see \cite{HO15}.

\section{Preliminaries}
In this section we collect results about $\ell^2$-Betti numbers and acylindrically hyperbolic groups used in the proof of Theorem \ref{main} and its corollaries.

An isometric action of a group $G$ on a metric space $S$ is called {\it acylindrical} if for every $\e>0$ there exist $R,N>0$ such that for every two points $x,y\in S$ satisfying $\d (x,y)\ge R$, there are at most $N$ elements $g\in G$ such that
$$
\d(x,gx)\le \e \;\;\; {\rm and}\;\;\; \d(y,gy) \le \e.
$$

By \cite[Theorem 1.2]{Osi13}, a group $G$ is acylindrically hyperbolic if and only if there exists a (possibly infinite) generating set $X$ of $G$ such that the Cayley graph of $G$ with respect to $X$, denoted $\Gamma (G,X)$, is hyperbolic, its Gromov boundary consists of more than $2$ points, and the action of $G$ on $\Gamma (G,X)$ is acylindrical.

For the definition of the first $\ell^2$-Betti number of a group we refer to  \cite{L}. Readers who are not familiar with this theory may think of the following L\"uck's Approximation Theorem as a definition in the particular case of finitely presented residually finite groups.

\begin{thm}[L\"uck]\label{Luck}
Let $G$ be a finitely presented group and let $G_1\ge G_2\ge \ldots $ be a chain of finite index normal subgroups of $G$ such that $\bigcap _{i=1}^\infty G_i=\{1\}$. Then $$\beta^{(2)}_1(G)=\lim \frac{b_1(G_i)}{[G:G_i]},$$ where $b_1(G_i)={\rm dim} (G_i/[G_i,G_i] \otimes \mathbb Q)$ is the ordinary first Betti number of $G_i$.
\end{thm}

Recall that a subgroup $H$ of a group $G$ is \emph{$s$-normal} if the intersection $H^g\cap H$ is infinite for every $g\in G$. The following theorem was proved by Peterson and Thom in \cite[Theorem 5.12]{PT}. In fact, they prove this result under a slightly weaker assumption that $H$ is weakly $s$-normal, which we do not need (and do not explain) here. Note also that the assumption that $H$ is infinite in the formulation of \cite[Theorem 5.12]{PT} is redundant since $s$-normal subgroups are infinite by the definition.

\begin{thm}[Peterson-Thom]\label{PT}
Let $G$ be a countable group, $H$ an infinite index $s$-normal subgroup of $G$.  Suppose that $\beta^{(2)}_1(H)<\infty $ (e.g., this is the case if $H$ is finitely generated). Then $\beta^{(2)}_1(G)=0$.
\end{thm}

We also need the following theorem, which provides a sufficient condition for an HNN-extension to be acylindrically hyperbolic.

\begin{thm}[Minasyan-Osin, {\cite[Crollary 2.2]{MO}}]\label{MO}
Let $G$ be an HNN-extension of a group $A$ with associated subgroups $C$ and $D$. Suppose that $C\ne A\ne D$ and there exists $g\in G$ such that $C^g\cap C$ is finite. Then $G$ is acylindrically hyperbolic.
\end{thm}

\section{Proofs}

Our proof of Theorem \ref{main} is based on the following elementary lemma.

\begin{lem}\label{HNN}
Let $G$ be a finitely presented group admitting an epimorphism $\e\colon G\to \mathbb Z$. Then $G$ is an HNN-extension of a finitely generated group with finitely generated associated subgroups.
\end{lem}

\begin{proof}
We fix a finite presentation
\begin{equation}\label{pres}
G=\langle t, a_1, \ldots, a_k\mid R_1, \ldots , R_m\rangle
\end{equation}
of the group $G$ such that $\e(t)$ generates $\mathbb Z$.
Note that the total sum of exponents of all occurrences of $t^{\pm 1}$ in each $R_i$ is $0$. Let $N$ be the maximal number of occurrences of $t$ in a relator in (\ref{pres}). Then every $R_i$ can be rewritten (possibly after conjugation by a suitable power of $t$) as a product of $b_{\alpha\beta}^{\pm 1}=t^\beta a_\alpha^{\pm 1} t^{-\beta}$, where the indices range as follows: $1\le \alpha\le k$, $0 \le \beta\le N$. Let $S_i$ be the word in the alphabet $B^{\pm 1}$, where
$$
B=\{b_{\alpha\beta}\mid 1\le \alpha\le k,\, 0 \le \beta\le N\},
$$
obtained from $R_i$ after such a rewriting. By using Tietze transformations, we can rewrite (\ref{pres}) in the form
$$
G=\langle t, B \mid S_1, \ldots, S_m, \mathcal T\rangle ,
$$
where $\mathcal T$ is the set of all relations of the form
$$tb_{\alpha, \beta}t^{-1} = b_{\alpha, \beta +1}$$
for $1\le \alpha\le k$ and $0 \le \beta\le N-1$.
In particular, $G$ is an HNN-extension of the subgroup generated by $B$ with finitely generated associated subgroups.
\end{proof}

\begin{proof}[Proof of Theorem \ref{main}]
Let $G$ be a group satisfying the assumptions of Theorem \ref{main}. By Lemma \ref{HNN}, $G$
splits as an HNN-extension of a finitely generated group $A$ with finitely generated associated  subgroups $C$, $D$ with a stable letter $t$.

We first observe that $C\ne A\ne D$. Indeed, assume $A=C$ or $A=D$. Let $\e\colon G\to \mathbb Z$ denote the natural homomorphism sending $t$ to a generator of $\mathbb Z$ and all elements of $A$ to $0$. Let $K_n$ denote the preimage of $n\mathbb Z$ under $\e$. Then $\beta^{(2)}_1(K_n)=n\beta_1^{(2)}(G)\to \infty $ as $n\to \infty$. On the other hand, it is straightforward to check that $K_n$ is generated by $A$ and $t^n$, which yields a uniform upper bound on the number of generators of $G$. Since the first $\ell^2$-Betti number of any group is bounded above by the number of generators minus $1$, we get a contradiction.

Thus $C\ne A\ne D$. By Theorem \ref{PT}, a countable group with positive first $\ell^2$-Betti number cannot contain finitely generated $s$-normal subgroups of infinite index. In particular, $C$ is not $s$-normal in $G$, i.e., there exists $g\in G$ such that $|g^{-1}Cg\cap C|<\infty$. Now applying Theorem \ref{MO} we conclude that $G$ is acylindrically hyperbolic.
\end{proof}

\begin{proof}[Proof of Corollary \ref{cor1}]
By Theorem \ref{Luck} every finitely presented residually finite group $G$ with $\beta_1^{(2)}(G)>0$ has a finite index subgroup that surjects onto $\mathbb Z$. Hence we can apply Theorem \ref{main}.
\end{proof}

\begin{proof}[Proof of Corollary \ref{cor2}]
Let $G$ be a group of deficiency $def(G)\ge 2$. Recall that $\beta_1^{(2)}(G)\ge def(G)-1$. Hence $G$ has positive first $\ell^2$-Betti number. Note also that $G$ obviously surjects onto $\mathbb Z$. It remains to apply Theorem \ref{main}.
\end{proof}

\begin{proof}[Proof of Corollary \ref{cor3}]
Parts (a)--(d) are well-known for acylindrically hyperbolic groups. For (a) and (c) see Theorem 2.33  and Theorem 2.35 in \cite{DGO}. Claim (b) is essentially due to Hamenst\" adt who proved it in different terms (see \cite{Ham} and the discussion before Theorem 8.3 in \cite{Osi13}). For (d) we refer to \cite{HO15}.
\end{proof}

We mention one more corollary. Recall that a group $G$ belongs to the class $\mathcal C_{reg}$ if $H_b^2(G, \ell^2(G))\ne 0$; further $G$ belongs to the class $\mathcal D_{reg}$ if $G$ is non-amenable and there exists an unbounded quasi-cocycle $G\to \ell^2(G)$. These classes were introduced by Monod-Shalom \cite{MS} and Thom \cite{T} in their work on rigidity of group actions. Currently no relation between these classes is known, although they seem likely to coincide. For a more detailed discussion we refer to \cite{T}.

It is well-known and fairly easy to see that every quasi-cocycle $G\to \ell^2(G)$ that is not a sum of a true cocycle and a bounded map gives rise to a non-trivial element of $H_b^2(G, \ell^2(G))$. Since $\beta_1^{(2)}(G)>0$ if and only if $G$ is non-amenable and there exists a non-zero cocycle $G\to \ell^2(G)$, we obtain the following alternative first proved by Thom \cite{T}: If $G\in \mathcal D_{reg}$, then either $G\in \mathcal C_{reg}$ or $\beta_1^{(2)} (G)>0$. From part (a) of the previous corollary and the L\"uck approximation, we immediately obtain the following.

\begin{cor}
Let $G\in \mathcal D_{reg}$ be a finitely presented group. Suppose that $G$ surjects onto $\mathbb Z$. Then $G\in \mathcal C_{reg}$.
\end{cor}

We conclude with two open questions.
\begin{prob}
Is every finitely presented group $G$ with positive first $\ell^2$-Betti number acylindrically hyperbolic?
\end{prob}

Note that there exist finitely presented groups with positive first $\ell^2$-Betti numbers that do not even virtually surject on $\mathbb Z$ (e.g., one can take $G=Q\ast Q$, where $Q$ is a finitely presented infinite simple group). Thus there is no way to reduce the general case of this problem to Theorem \ref{main}.

The second question is motivated by the fact that the use of $\ell^2$-Betti numbers and the Peterson-Thom theorem in our proof is unavoidable, even if we only want to prove our main theorem for groups of deficiency at least $2$. It would be interesting to find an alternative (group theoretic) proof in this particular case. The main difficulty here is the following.

\begin{prob}
Find a group theoretic argument showing that every finitely generated $s$-normal subgroup of a finitely presented group of deficiency at least $2$ is of finite index.
\end{prob}

\vspace{1cm}

\noindent {\bf D. Osin:  } Stevenson Center 1326, Department of Mathematics, Vanderbilt University, Nashville, TN 37240, USA\\
email: {\it denis.osin@gmail.com}
\end{document}